\documentclass[12pt,reqno]{amsart}

\usepackage{amssymb,latexsym}

\usepackage{enumerate}
\allowdisplaybreaks
\usepackage[french,english]{babel}
\usepackage{amsmath}
\usepackage{graphicx}
\usepackage{amssymb}
\usepackage{bbm}
\usepackage{amsthm,mathtools}
\usepackage{ulem}
\usepackage{geometry}
\usepackage{tikz-cd}
\usepackage{mathrsfs}
\usepackage[colorinlistoftodos]{todonotes}
\usepackage{enumitem}
\usepackage{verbatim}
\usepackage[foot]{amsaddr}
\usepackage{dsfont}
\usepackage{cite}
\usepackage[T1]{fontenc}

\makeatletter

\@namedef{subjclassname@2010}{
	
	\textup{2020} Mathematics Subject Classification}

\makeatother
\newtheorem{thm}{Theorem}[section]
\newtheorem*{thm*}{Theorem}

\newtheorem{lem}[thm]{Lemma}

\theoremstyle{definition}

\numberwithin{equation}{section}

\newcommand{\bg}{\big}

\newcommand{\bgg}{\bigg}

\newcommand{\Bg}{\Big}

\newcommand{\lo}{\log_2}
\newcommand{\lt}{\log_3}
\newcommand{\inv}{^{-1}}
\newcommand{\ph}{\Phi}

\newcommand{\mca}{\mathcal{A}}

\newcommand{\mmd}{\mathrm{d}}
\newcommand{\mme}{\mathrm{e}}

\newcommand{\mit}{\mathrm{i}t}
\newcommand{\mdt}{\mathrm{d}t}

\newcommand{\whp}{\widehat{\Phi}}

\usepackage{hyperref}
\hypersetup{colorlinks=true,linkcolor=blue,anchorcolor=blue,citecolor=blue}
\usepackage{color}
\newcommand{\newabstract}[1]{%
	\par\bigskip
	\csname otherlanguage*\endcsname{#1}%
	\csname captions#1\endcsname
	\item[\hskip\labelsep\scshape\abstractname.]
}

\begin{document}

	\baselineskip=17pt

	\title[An Omega result for partial sums of the Riemann zeta function]{An Omega result for partial sums of the Riemann zeta function}

	\author{Qiyu Yang\textsuperscript{1}}
    \author{Shengbo Zhao\textsuperscript{2}}
	\address{1. School of Mathematics and Statistics, Henan Normal University, Xinxiang 453007, CHINA}
	\address{2. School of Mathematical Sciences, Key Laboratory of Intelligent Computing and Applications (Tongji University), Ministry of Education, Tongji University, Shanghai 200092, China}
	\email{qyyang.must@gmail.com}
	\email{shengbozhao@hotmail.com}

	\begin{abstract} 
	   In this paper, we establish an Omega result for partial sums of the Riemann zeta function on the \(1\)-line. Our work relies on the long resonance method and an effective asymptotic formula for smooth numbers. For sufficiently long partial sums, our result approaches the work obtained by Aistleitner, Mahatab and Munsch in 2019.
	\end{abstract}
	
	
	
	\maketitle

\section{Introduction}

Partial sums of the Riemann zeta function \(\zeta(s)\) play a fundamental role in analytic number theory, as they provide finite approximations to the Riemann zeta function and naturally arise in the study of Dirichlet polynomials. In many regions of the complex plane, \(\zeta(s)\) can be effectively approximated by its partial sums with a controllable error term. Consequently, the analytic behavior of these partial sums often reflects that of the Riemann zeta function itself. In particular, the study of Omega results for partial sums not only sheds light on extreme values of the Riemann zeta function but also provides insight into related problems concerning Dirichlet polynomials and multiplicative functions. 

Another motivation for studying Omega results for partial sums comes from the work of Gonek and Ledoan in \cite{gonek2010IMRN}. For \(N \ge 2\), by studying the distribution of the zeros of 
\[
\zeta_N(s) = \sum_{n \le N}\frac{1}{n^s},
\]
they obtained a zero-counting formula and a zero-density estimate to the right of the critical line. These results reflect, to some extent, the behavior of the zeros of \(\zeta(s)\). Since Omega results for \(\zeta(s)\) are closely related to the distribution of its zeros, it is natural to investigate analogous Omega results for \(\zeta_N(s)\). For further work on partial sums of the Riemann zeta function, we recommend \cite{BS2018BLMS,levinson1973USA,Roy2019Adv} and the references therein.

Here and throughout this paper, we write \(\log_j\) for the \(j\)-th iterated logarithm. Then, we have the following result:
\begin{thm}
    \label{thm1}
    Let \( 0<\theta<1\) and \(B \ge 1\) be fixed. For sufficiently large \(T\), let \(N\) satisfy
    \[
    2 \le N \le (\log T \lo T)^B.
    \]
    Then, uniformly for \(N\) in this range, we have
    \[
    \max_{T^\theta \le t \le T}|\zeta_N(1+\mit)| \ge  (\lo T+\lt T) \int_{0}^{\frac{\log N}{\lo T+\lt T}} \rho(u) \mmd u- C_{\theta,B}.
    \]
    Here, \(\rho(u)\) denotes the Dickman function, and \(C_{\theta,B}>0\) is a constant depending only on \(\theta\) and \(B\).
\end{thm}

In 2019, Aistleitner, Mahatab and Munsch \cite{aistleitner2019IMRN} proved that there is a constant \(C\) such that
\[
\max_{\sqrt{T} \le t \le T} |\zeta(1+\mit)| \ge \mme^\gamma(\lo T+\lt T-C)
\]
for sufficiently large \(T\), where \(\gamma\) is the Euler-Mascheroni constant. Their work refines the earlier result of Levinson \cite{levinson1972AA}, and the work of Granville and Soundararajan \cite{Gran06extreme}. Furthermore, when \(N= (\log T \lo T)^B\) with \(B\) sufficiently large, our Theorem \ref{thm1} implies that for every \(\varepsilon>0\), one has
\[
\max_{T^\theta \le t \le T}|\zeta_N(1+\mit)| \ge (\mme^\gamma-\varepsilon)(\lo T+\lt T) - C_{\theta,B}.
\]
Thus, compared with the result of \cite{aistleitner2019IMRN}, sufficiently long partial sums recover the main order of magnitude of extreme values of the Riemann zeta function on the \(1\)-line.

On the other hand, if we set \(B=1\) in Theorem \ref{thm1}, 
\[
\max_{T^\theta \le t \le T}|\zeta_N(1+\mit)| \ge \log N -C_\theta
\]
holds uniformly for \(N \in [2, \log T \lo T]\). This reduces Theorem \ref{thm1} to a simpler form.

Our proof of Theorem \ref{thm1} is based on the resonance method. Since its introduction by Soundararajan \cite{soundararajan2008extreme}, this method has been widely used in the study of Omega results. It has subsequently been refined by Aistleitner \cite{aistleitner2016MathAnn}, Bondarenko and Seip \cite{bondarenko2018argument}, and others, leading to sharper Omega results. For further results and details concerning the resonance method, we recommend \cite{Aistleitner2019QJMath,bondarenko2023dichotomy,tenen2019galsum} and the references therein.

We organize this paper as follows. In Section \ref{sec-pre}, we present some preliminary work, including an important lemma and the construction of the resonator. In Section \ref{sec-proof-thm1}, we prove Theorem \ref{thm1}. Finally, in Section \ref{sec-final}, we give some comments.

\section{Preliminaries}
\label{sec-pre}

In this section, we carry out some preliminary work.  Let \(P^+(n)\) denote the largest prime factor of \(n\), and \(P^+(1)=1\). Then let \(\Psi(x,y)\) count the number of integers \(n\) not exceeding \(x\) with prime factors at most \(y\), that is,
\[
 \Psi(x,y) := \sum_{\substack{n \le x \\ P^+(n) \le y}} 1.
\]
First, we state the following approximate formula for \(\Psi(x,y)\):
\begin{lem}
\label{lem1}
    Let \(x \ge y \ge 2\), and put \(u =\log x/ \log y\). Then, the following estimate holds uniformly
    \[
    \Psi(x,y) = x \rho(u)+ O\Bg(\frac{x}{\log y}\Bg).
    \]
    Here, \(\rho(u)\) denotes the Dickman function.
\end{lem}

\begin{proof}
    This is \cite[Section 5.3, Theorem 6]{tenenbaumbook}.
\end{proof}

Then, we construct the resonator \(R(t)\) in order to apply the long resonance method in \cite{aistleitner2019IMRN}. Let \(T\) be sufficiently large, and set \(X=\eta \log T \lo T\), where \(\eta = (1-\theta)/4 >0\). Define \(r(n)\) to be a completely multiplicative function whose values at primes \(p\) are given by
\begin{equation*}
    r(p) =
    \begin{cases}
        1-\frac{p}{X}, ~ &\operatorname{if} p \le X, \\
		0, ~ &\operatorname{if} p > X.
    \end{cases}
\end{equation*}
Furthermore, we define the resonator \(R(t)\) as follows
\[
R(t) = \sum_{n =1}^{\infty}r(n)n^{\mit} = \prod_{p \le X}\bg(1-r(p)p^{\mit} \bg)\inv.
\]
The prime number theorem yields that
\[
\log |R(t)| \le \sum_{p \le X}(\log X - \log p) = \bg(1+o(1)\bg)\frac{X}{\log X}.
\]
Thus, using the definition of \(X\), we have
\begin{align}
    \label{Rupper}
    |R(t)|^2 \le T^{2\eta+o(1)}.
\end{align}

\section{Proof of Theorem \ref{thm1}}
\label{sec-proof-thm1}

In this section, we prove Theorem \ref{thm1}. As in \cite{bondarenko2017Duke}, let \(\ph(t) := \mme^{-t^2/2}\) be the Gaussian function. Then \(\ph\) and its Fourier transform \(\whp(\xi)=\sqrt{2\pi}\Phi(\xi)\) are both positive. Since \(\ph\) decays exponentially, all relevant sums and integrals below are absolutely convergent, allowing us to interchange summation and integration freely. To apply the resonance method, we define the following four integrals:
\begin{align*}
    M_1 &:= M_1(R,T) = \int_{T^\theta}^T |R(t)|^2 \ph\Bg(\frac{t\log T}{T} \Bg) \mdt, \\
    M_2 &:= M_2(R,T) = \int_{T^\theta}^T \zeta_N(1+\mit)|R(t)|^2 \ph\Bg(\frac{t\log T}{T} \Bg) \mdt, \\
    I_1 &:= I_1(R,T) = \int_{-\infty}^{\infty} |R(t)|^2 \ph\Bg(\frac{t\log T}{T} \Bg) \mdt, \\
    I_2 &:= I_2(R,T) = \int_{-\infty}^{\infty} \zeta_N(1+\mit)|R(t)|^2 \ph\Bg(\frac{t\log T}{T} \Bg) \mdt,
\end{align*}
where \(0<\theta<1\). Trivially, we have
\[
\max_{T^\theta \le t \le T}|\zeta_N(1+\mit)| \ge \frac{|M_2|}{M_1}.
\]

To utilize properties of the Fourier transform of \(\ph\), we extend integrals \(M_1\) and \(M_2\) to the real axis. Since \(N \le (\log T \lo T)^B\), we have
\begin{align}
    \label{partialupper}
    |\zeta_N(1+\mit)| \le \sum_{n \le N}\Bg|\frac{1}{n^{1+\mit}}\Bg| \ll \log N \ll T^{o(1)}.
\end{align}
Thus, combining \eqref{Rupper} and \eqref{partialupper} shows that
\begin{align}
    \label{error1}
    \bgg|\int_{|t| < T^\theta} \zeta_N(1+\mit)|R(t)|^2 \ph\Bg(\frac{t\log T}{T} \Bg) \mdt\bgg| \ll T^{\theta+2\eta+o(1)}.
\end{align}
For the tail integral, the rapid decay of \(\ph\) gives 
\begin{align}
    \label{error2}
    \bgg|\int_{|t| > T} \zeta_N(1+\mit)|R(t)|^2 \ph\Bg(\frac{t\log T}{T} \Bg) \mdt\bgg| \ll 1.
\end{align}
Using \eqref{error1} and \eqref{error2} shows that 
\begin{align}
    \label{M2I2}
    2M_2 = I_2 + O\bg(T^{\theta+2\eta+o(1)}\bg).
\end{align}
Similarly, 
\begin{align}
    \label{M1I1}
    2M_1 = I_1 + O\bg(T^{\theta+2\eta+o(1)}\bg).
\end{align}
On the other hand, expanding \(|R(t)|^2\) in \(I_1\), we have
\[
I_1 = \int_{-\infty}^\infty \sum_{m,n=1}^\infty r(m)r(n) \Bg(\frac{m}{n}\Bg)^{\mit}\ph\Bg(\frac{t\log T}{T} \Bg) \mdt.
\]
Since \(r(n)\) is completely multiplicative, \(r(1)=1\). Together with the positivity of \(\whp\), this gives
\begin{align}
    \label{I1lower}
    I_1 \gg T^{1+o(1)}.
\end{align}
Combining \eqref{M2I2}, \eqref{M1I1} and \eqref{I1lower} implies that
\begin{align}
    \label{max}
    \max_{T^\theta \le t \le T}|\zeta_N(1+\mit)| \ge \frac{I_2}{I_1} + O\bg(T^{\theta+2\eta-1+o(1)}\bg).
\end{align}
With \(\eta = (1-\theta)/4\), the error term on the right-hand side of \eqref{max} is small.

We now derive a lower bound for the ratio \(I_2/I_1\). To this end, expanding both \(|R(t)|^2\) and \(\zeta_N(1+\mit)\) in \(I_2\) yields that
\begin{align*}
    I_2 &= \int_{-\infty}^{\infty} \sum_{k \le N}\frac{1}{k^{1+\mit}} \sum_{m,n \ge 1}^\infty r(m)r(n)\Bg(\frac{m}{n}\Bg)^{\mit}\ph\Bg(\frac{t\log T}{T} \Bg) \mdt\\
    &= \sum_{k \le N}\frac{1}{k}  \sum_{m,n \ge 1}^\infty r(m)r(n) \int_{-\infty}^{\infty}\Bg(\frac{m}{kn}\Bg)^{\mit}\ph\Bg(\frac{t\log T}{T} \Bg) \mdt.
\end{align*}
The positivity of \(\whp\) allows us to restrict to the terms with \(k \mid m\). More precisely, writing \(m = k\ell\), and using the definition of \(I_1\), we obtain that
\begin{align*}
    I_2 & \ge \sum_{k \le N}\frac{1}{k} \sum_{\ell,n \ge1}^\infty r(k\ell)r(n)\int_{-\infty}^{\infty}\Bg(\frac{k\ell}{kn}\Bg)^{\mit}\ph\Bg(\frac{t\log T}{T} \Bg) \mdt \\
    & = \sum_{k \le N}\frac{r(k)}{k} \sum_{\ell,n \ge1}^\infty r(\ell)r(n)  \int_{-\infty}^{\infty}\Bg(\frac{\ell}{n}\Bg)^{\mit}\ph\Bg(\frac{t\log T}{T} \Bg) \mdt \\
    & = \sum_{k \le N}\frac{r(k)}{k} \cdot I_1.
\end{align*}
Here, we use the fact that \(r(n)\) is a completely multiplicative function. Thus, it follows that
\begin{align}
    \label{I2I1lower}
    \frac{I_2}{I_1} \ge \sum_{k \le N}\frac{r(k)}{k} =: D_N.
\end{align}

Since \(r(k)=0\) whenever \(P^+(k)>X\), we may write
\[
D_N=\sum_{\substack{k\le N\\ P^+(k)\le X}}\frac{r(k)}{k}.
\]
Define
\begin{align*}
    \mca(x,y) := \sum_{\substack{n \le x \\ P^+(n) \le y}} r(n). 
\end{align*}
By partial summation, we have
\begin{align}
    \label{DNeq}
    D_N=\frac{\mca(N,X)}{N}+ \int_1^N\frac{\mca(t,X)}{t^2}\mdt.
\end{align}

We next show that, when \(X\) is sufficiently large,
\begin{align}
    \label{APapproximate}
    \mca(t,X) = \Psi(t,X)\Bg(1+O_B\Bg(\frac{1}{\log X}\Bg) \Bg)
\end{align}
holds uniformly for \(2 \le t \le X^B\). By the definition of the completely multiplicative function \(r(n)\),
for every prime \(p\le X\),
\[
0\le r(p)=1-\frac{p}{X}\le 1.
\]
Hence,
\[
0\le r(k)\le 1
\]
whenever \(P^+(k)\le X\), and therefore
\[
0\le \mca(t,X)\le \Psi(t,X).
\]
Moreover,
\begin{align}
\label{PminusA}
    \Psi(t,X) - \mca(t,X) = \sum_{\substack{k \le t \\ P^+(k) \le X}} (1-r(k)).
\end{align}
Write
\[
k = \prod_{p \le X} p^{\nu_p(k)},
\]
where \(\nu_p(k)\) denotes the exponent of the prime \(p\) in the prime factorization of \(k\). Since \(r(n)\) is completely multiplicative, we have
\[
r(k) = \prod_{p \le X}\Bg(1-\frac{p}{X}\Bg)^{\nu_p(k)}.
\]
Using the following classical inequality
\begin{align}
    \label{classical1}
        1 - \prod_j (1-x_j) \le \sum_j x_j, \quad 0 \le x_j \le 1,
\end{align}
together with
\begin{align}
    \label{classical2}
    1-(1-x)^\alpha \le \alpha x, \quad 0 \le x \le 1,
\end{align}
yields that
\[
1-r(k) \le \frac{1}{X}\sum_{p \le X}\nu_p(k)p.
\]
Consequently, we obtain
\begin{align}
    \label{PminusA2}
    0 \le \Psi(t,X) - \mca(t,X) \le \frac{1}{X}\sum_{\substack{k\le t \\ P^+(k) \le X}} \sum_{p \le X} \nu_p(k)p =  \frac{1}{X} \sum_{p \le X}p \sum_{\substack{k\le t \\ P^+(k) \le X}}\nu_p(k).
\end{align}
Note that
\[
\nu_p(k) = \sum_{\beta \ge 1} \mathds{1}_{p^\beta \mid k},
\]
where \(\mathds{1}\) indicates the indicator function. Therefore, the inner sum on the right-hand side of \eqref{PminusA2} is bounded by
\[
\sum_{\substack{k\le t \\ P^+(k) \le X}}\sum_{\beta \ge 1} \mathds{1}_{p^\beta \mid k} =\sum_{\beta \ge 1} \sum_{\substack{p^\beta \ell \le t \\ P^+(p^\beta \ell) \le X}} 1 \le \sum_{\substack{p^\beta \ell \le t \\ P^+(\ell) \le X}} 1 =\sum_{\beta \ge 1} \Psi\Bg(\frac{t}{p^\beta},X\Bg).
\]
Plugging this into \eqref{PminusA2} gives 
\begin{align}
    \label{PminusA3}
    0 \le \Psi(t,X) - \mca(t,X) \le \frac{1}{X} \sum_{p \le X}p \sum_{\beta \ge 1} \Psi\Bg(\frac{t}{p^\beta},X\Bg).
\end{align}

For \(2\le t\le X\), we have
\[
\Psi\left(\frac{t}{p^\beta},X\right) = \Bg\lfloor\frac{t}{p^\beta}\Bg\rfloor,
\]
thus, 
\[
\Psi(t,X)-A(t,X) \le\frac{t}{X}\sum_{p\le t}p\sum_{\beta\ge1}\frac1{p^\beta} 
= \frac{t}{X}\sum_{p \le t}\frac{p}{p-1} \ll \frac{t}{\log X}.
\]
Since \(\Psi(t,X) = \lfloor t \rfloor \asymp t\) as \(t \le X\), it follows that
\begin{align}
    \label{final1}
    \Psi(t,X)-A(t,X)\ll\frac{\Psi(t,X)}{\log X}.
\end{align}

On the other hand, for \(X<t\le X^B\), set
\[
v=\frac{\log t}{\log X}\in(1,B].
\]
It follows from Lemma \ref{lem1} that
\begin{align}
    \label{dickman}
    \Psi(t,X) = t\rho(v) \Bg(1+O_B\Bg(\frac{1}{\log X}\Bg)\Bg)
\end{align}
holds uniformly in this range. Since \(1 < v \le B\), we have \(\rho(v) \ge C >0\) for some constant \(C\). Thus, \(\Psi(t,X) \asymp_B t\). Combining this with \eqref{PminusA3} and \eqref{dickman}, we obtain
\[
\Psi(t,X) - \mca(t,X) \ll_B \frac{\Psi(t,X)}{X} \sum_{p \le X}p \sum_{\beta \ge 1}\frac{1}{p^\beta} = \frac{\Psi(t,X)}{X}\sum_{p \le X} \frac{p}{p-1}.
\]
Applying the prime number theorem, we have
\begin{align}
    \label{final2}
    \Psi(t,X) - \mca(t,X) \ll_B\frac{\Psi(t,X)}{\log X}.
\end{align}
It follows from \eqref{final1} and \eqref{final2} that, for all sufficiently large \(X\), \eqref{APapproximate} holds uniformly for \(1 \le t \le X^B\). 

We now derive a uniform approximation for \(\mca(t,X)\). When \(1\le t\le X\), since \(\rho(u)=1\) for \(0\le u\le1\), we have
\[
\Psi(t,X)=\lfloor t\rfloor=t\rho\Bg(\frac{\log t}{\log X}\Bg)+O(1).
\]
Furthermore, when \(X<t\le X^B\), \eqref{dickman} gives that 
\[
\Psi(t,X)=t\rho\Bg(\frac{\log t}{\log X}\Bg)+O_B\Bg(\frac{t}{\log X}\Bg).
\]
Hence, uniformly for \(1\le t\le X^B\), we have
\begin{align}
    \label{Aapproximate}
    \mca(t,X)=t\rho\Bg(\frac{\log t}{\log X}\Bg)+O_B\Bg(\frac{t}{\log X}+1\Bg).
\end{align}

Substituting \eqref{Aapproximate} into \eqref{DNeq}, we obtain
\begin{align}
    \label{Dfinal}
    D_N = \log X\int_0^{\frac{\log N}{\log X}}\rho(u) \mmd u +O_B(1),
\end{align}
since \(N \le X^B\).

Combining \eqref{max}, \eqref{I2I1lower} and \eqref{Dfinal}, we conclude that
\[
\max_{T^\theta\le t\le T}|\zeta_N(1+\mit)| \ge \log X \int_0^{\frac{\log N}{\log X}}\rho(u) \mmd u +O_B(1).
\]
Since \(X=\eta\log T \lo T\), we complete the proof of Theorem \ref{thm1}.

\section{Final comments}
\label{sec-final}

For fixed \(\sigma \in (1/2,1)\) and sufficiently large \(T\), Aistleitner \cite{aistleitner2016MathAnn} showed that 
\[
\max_{0 \le t \le T} |\zeta(\sigma+\mit)| \ge \exp\Bg(\frac{c_\sigma(\log T)^{1-\sigma}}{(\lo T)^\sigma} \Bg),
\]
where \(c_\sigma=0.18(2\sigma-1)^{1-\sigma}\). In fact, by applying the long resonance method from \cite{yang2023omega} and following an argument similar to that used in Section \ref{sec-proof-thm1}, one can obtain an Omega result for partial sums of the Riemann zeta function in the critical strip. More precisely, for \(N\) in the same range as in Theorem \ref{thm1}, we have
    \[
    \max_{T^\theta \le t \le T}|\zeta_N(\sigma+\mit)| \ge \frac{N^{1-\sigma}}{1-\sigma}\rho \Bg(\frac{\log N}{\lo T+\lt T} \Bg) + O_{\sigma,\theta,B}\Bg(\frac{N^{1-\sigma}}{\lo T} +1\Bg).
    \]
Since \(N \le X^B\), this resonator can capture only part of the full Euler product of the Riemann zeta function. Consequently, the above result is significantly weaker than the corresponding Omega result of Aistleitner \cite{aistleitner2016MathAnn}. A sharper bound would require \(N\) to reach at least \(\exp (c(\log T \lo T)^{1-\sigma})\) for some constant \(c>0\). In this range, however, estimates involving the smooth number counting function \(\Psi(x,y)\) may no longer be applicable. Overcoming this difficulty appears to be an interesting problem for further study.

\section*{Acknowledgments}
Qiyu Yang was supported by the Natural Science Foundation of Henan Province (Grant No. 252300421782). 
	
	\bibliographystyle{siam}
    \bibliography{reference}
\end{document}